\documentclass[12pt,reqno]{amsart}
\usepackage{mathtools}
\usepackage{amssymb, mathabx, enumitem}
\usepackage{color}
\usepackage[colorlinks=true,linkcolor=blue,urlcolor=blue,citecolor=blue, pagebackref]{hyperref}
\usepackage[alphabetic]{amsrefs}
\usepackage{tikz-cd}
\usepackage{tikz}
\usepackage{tikz-3dplot}
\usetikzlibrary{arrows}
\usepackage{graphics}
\usepackage{arydshln}

\usepackage[margin=1in]{geometry}

\newcommand{\C}{\mathbb{C}}

\newcommand{\Z}{\mathbb{Z}}
\newcommand{\N}{\mathbb{N}}

\newcommand{\op}{\operatorname}

\newtheorem{theorem}{Theorem}[section]

\newtheorem{corollary}[theorem]{Corollary}
\newtheorem{question}[theorem]{Question}
\newtheorem{proposition}[theorem]{Proposition}

\newtheorem{definition}[theorem]{Definition}
\newtheorem{definition/lemma}[theorem]{Definition/Lemma}

\title[GIT and stretched Kostka quasi-polynomials]{Geometric invariant theory and stretched Kostka quasi-polynomials}
\author{Marc Besson}
\address{YMSC, Tsinghua University, Beijing, P.R. China, 100871}
\email{bessonm@tsinghua.edu.cn}
\author{Sam Jeralds}
\address{University of Sydney, Camperdown, NSW 2006}
\email{samuel.jeralds@sydney.edu.au}
\author{Joshua Kiers}
\address{Epic Systems Corporation, Verona, WI 53593}
\email{jkiers@epic.com}

\begin{document}
\maketitle
\begin{abstract}
For $G$ a semisimple, simply-connected complex algebraic group and two dominant integral weights $\lambda, \mu$, we consider the dimensions of weight spaces $V_\lambda(\mu)$ of weight $\mu$ in the irreducible, finite-dimensional highest weight $\lambda$ representation. For natural numbers $N$, the function $N \mapsto \dim V_{N\lambda}(N\mu)$ is a quasi-polynomial in $N$, the stretched Kostka quasi-polynomial. Using methods of geometric invariant theory (GIT), we realize the degree of this quasi-polynomial as the dimension of a certain GIT quotient. As a result, we resolve a conjecture of Gao and Gao on an explicit formula for this degree. We also discuss periods of this quasi-polynomial determined by the GIT approach, and give computational evidence supporting a geometric determination of the minimal period. 
\end{abstract}

\section{Introduction} \label{Intro}

\subsection{Stretched Kostka quasi-polynomials and their degrees}
Fix a positive natural number $n$ and two partitions $\lambda, \mu \vdash n$. For two such partitions, the Kostka number $K_{\lambda, \mu}$ plays a central role in algebraic combinatorics and its interaction with representation theory. Among its many formulations, $K_{\lambda, \mu}$ encodes the number of semistandard Young tableaux of shape $\lambda$ and content $\mu$, the decomposition of the Schur function $s_\lambda$ into the monomial basis, the decomposition of permutation modules for the symmetric group $S_n$ into Specht modules, and dimensions of weight spaces in irreducible polynomial representations of general linear groups, to name a few incarnations.  

While the values of Kostka numbers can vary wildly depending on the partitions $\lambda, \mu$, the nonvanishing of $K_{\lambda, \mu}$ is captured by a textbook combinatorial criterion: $K_{\lambda, \mu} \neq 0 \iff \mu \trianglelefteq \lambda$, where $\trianglelefteq$ denotes the \emph{dominance order} on partitions (cf. \cite{Ful}*{\S 2.2}). Among its many features, the dominance order is preserved under scaling; that is, for any integer partitions $\lambda, \mu$ and $N \geq 1$ we have that $K_{\lambda, \mu} \neq 0 \iff K_{N\lambda, N\mu} \neq 0$. A goal, then, in studying the behavior of Kostka numbers is to understand the function $K_{\lambda, \mu}(\cdot): \N \to \N$ given by 
$$
K_{\lambda, \mu}(N):= K_{N\lambda, N\mu};
$$
this is a polynomial, as proven in \cite{KR}, which we will refer to as the \emph{stretched Kostka polynomial}. Using the realization of $K_{\lambda, \mu}(N)$ as the Ehrhart polynomial of a certain Gelfand--Tsetlin polytope $GT_{\lambda, \mu}$, McAllister \cite{McAll} computed the degree of $K_{\lambda, \mu}(N)$, resolving a conjecture of King--Tollu--Toumazet \cite{KTT}. 

Now, let $G$ be a semisimple, simply-connected complex algebraic group, and fix a choice of maximal torus and Borel subgroup $T \subset B \subset G$. Associated to this choice, we get a root system $\Phi$ with positive roots $\Phi^+$, and integral characters (or weights) $X^\ast(T)$ with dominant integral characters (or weights) $X^\ast(T)^+$. To any $\lambda \in X^\ast(T)^+$, we can associate to $G$ the finite-dimensional, irreducible representation $V_\lambda$ with highest weight $\lambda$. For $\mu \in X^\ast(T)^+$, we can consider the $\mu$ weight space of $V_\lambda$, which we denote by $V_\lambda(\mu)$. Generalizing the realization of the Kostka numbers as dimensions of weight spaces for representations of general linear groups, we set 
$$
K^G_{\lambda, \mu} = K_{\lambda, \mu}= \dim(V_\lambda(\mu));
$$
we abuse notation and drop the dependence on $G$. Like the classical Kostka numbers, there is a straightforward criterion for when $K_{\lambda, \mu} \neq 0$ given by the dominance order on weights, whereby
$$
K_{\lambda, \mu} \neq 0 \iff \mu \leq \lambda \iff \lambda-\mu \in Q^+:= \Z_{\geq 0}(\Phi^+).
$$ 
This condition again is invariant under scaling, and we can consider the function $K_{\lambda, \mu}(\cdot): \N \to \N$ given by 
$$
K_{\lambda, \mu}(N):= K_{N\lambda, N\mu}=\dim V_{N\lambda}(N\mu).
$$
Unlike the type A case, $K_{\lambda, \mu}(N)$ is now in general a \emph{quasi-polynomial} in $N$:

\begin{definition}
A function $f:\N\to \N$ is a \emph{quasi-polynomial} if there exist a positive integer $d \geq 1$ and polynomials $p_0, p_1, \dots, p_{d-1}$ such that $f(N) = p_i(N)$ for all $N \equiv i \  \mod d$.
We call such an integer $d$ a \emph{period} of the quasi-polynomial. Note that $d$ need not be minimal. 

The \emph{degree} of $f$ is the highest degree of its constituent polynomials $p_0, p_1, \dots, p_{d-1}$. 
\end{definition}

We call $K_{\lambda,\mu}(N)$ the \emph{stretched Kostka quasi-polynomial}. 

Again using Ehrhart theory, now applied to Berenstein--Zelevinsky polytopes $BZ_{\lambda, \mu}$, Gao and Gao \cite{GG} extended McAllister's work and gave a uniform formula for the degree of $K_{\lambda, \mu}(N)$ for classical types, as follows.

\begin{theorem}{\cite{GG}*{Theorem 1.2}}\label{GGClassical} 
Let $G$ be a complex, semisimple, simply-connected algebraic group of classical type with root system $\Phi$, and let $\lambda, \mu \in X^\ast(T)^+$ be two dominant weights such that $\lambda-\mu=\sum_i c_i \alpha_i$ with each $c_i \in \mathbb{Z}_{\geq 0}$, in terms of the simple roots $\alpha_i$ of $G$. Also write $\lambda=\sum_i d_i \varpi_i$ in terms of the fundamental weights. Then the degree of $K_{\lambda, \mu}(N)$ is 
$$
\frac{1}{2} |\Phi^{(1)}|-\mathrm{rank}(\Phi^{(1)})-\frac{1}{2}|\Phi^{(2)}|,
$$
where $\Phi^{(1)}$ and $\Phi^{(2)}$ are the root subsystems determined by 
$$
\Phi^{(1)} = \mathrm{span}\{ \alpha_i | c_i \neq 0\}, \ \Phi^{(2)} = \mathrm{span}\{ \alpha_i | c_i \neq 0, \ d_i =0\}.
$$
\end{theorem}

\noindent They further conjectured that this degree formula holds for arbitrary semisimple $G$ \cite{GG}*{Conjecture 1.3}. The first goal of the present work is to prove this conjecture.

As part of their proof for classical types, Gao and Gao make use the following special subclasses of pairs of dominant integral weights $(\lambda, \mu)$, defined in \cite{GG}*{\S 2.1}. 

\begin{definition} For any semisimple, simply-connected $G$, a pair of dominant integral weights $(\lambda, \mu)$ with $\lambda-\mu=\sum_{i} c_i \alpha_i$ is called 
\begin{enumerate}
\item \emph{primitive} if $c_i \neq 0$ for all $i$, and
\item \emph{simple primitive} if the pair is primitive and group $G$ is simple. 
\end{enumerate}
\end{definition}

A crucial observation is that, to compute the degree of $K_{\lambda, \mu}(N)$, one can reduce to the case when $(\lambda, \mu)$ is simple primitive \cite{GG}*{Proposition 3.4, Corollary 3.5}. With this reduction, the formula for the degree of $K_{\lambda, \mu}(N)$ simplifies to the following, which confirms their conjecture (Theorem \ref{MT1Restated} in the text).

\begin{theorem} \label{MainTheorem1}
Let $G$ be a simple, simply-connected complex algebraic group, and $(\lambda, \mu)$ a primitive pair of dominant integral weights. Write $\lambda = \sum_i d_i \varpi_i$ in the basis of fundamental weights, with $d_i \geq 0$. Then $K_{\lambda, \mu}(N)$ is a quasi-polynomial of degree
$$
\frac{1}{2}|\Phi|-\frac{1}{2}|\Phi'| - \mathrm{rank}(\Phi),
$$
where $\Phi'$ is the sub-root system given by $\mathrm{span}\{ \alpha_i | d_i = 0\} \subseteq \Phi$.
\end{theorem}

\begin{corollary}
Theorem \ref{GGClassical} holds for all semisimple $G$. 
\end{corollary}

\subsection{Geometric Invariant Theory and quasi-polynomiality} \label{GITIntro}
We assemble a uniform proof of Theorem \ref{MainTheorem1} using techniques from geometric invariant theory (GIT). The application of GIT methods to representation-theoretic questions is now well-established in the literature, with particular focus on problems involving the branching of representations and the behavior of weight multiplicities. Without any hope of being exhaustive, we highlight \cite{Kumar2}, \cite{Belkale}, \cite{BJK}, \cite{Sherm}, and references therein for examples of these techniques and related results. 

Broadly, the proof of Theorem \ref{MainTheorem1} consists of connecting the degree of the stretched Kostka quasi-polynomial $K_{\lambda, \mu}(N)$ to the dimension of a certain GIT quotient of a partial flag variety $G/P_\lambda$. The restriction to pairs $(\lambda, \mu)$ which are simple primitive greatly simplifies this dimension computation and gives a geometric interpretation to the degree formulas conjectured in \cite{GG}.

Further, the geometric perspective afforded by the GIT constructions gives rise not just to information on the degree of $K_{\lambda, \mu}(N)$ but also to its periods. In particular, for simple primitive pairs $(\lambda, \mu)$ and the corresponding line bundle $\mathbb{L}$ (cf. Equation (\ref{linebundef})) there is some minimal integer $d \geq 1$ such that $\mathbb{L}^{\otimes d}$ \emph{descends} to the GIT quotient. As an easy observation during the proof of Theorem \ref{MainTheorem1}, we can conclude the following proposition.

\begin{proposition} \label{MainTheorem2}
Let $G$ be a simple, simply-connected complex algebraic group, and let $(\lambda, \mu)$ be a primitive pair of dominant integral weights.  Let $d \geq 1$ be minimal such that $\mathbb{L}^{\otimes d}$ descends to the GIT quotient. Then $d$ is a period of the stretched Kostka quasi-polynomial $K_{\lambda, \mu}(N)$. 
\end{proposition}

In Section \ref{Descent-Examples}, we discuss how this integer $d$ should be determinable from the data of $G$, $\lambda$, and $\mu$, and using a result of \cite{Kumar1} give an explicit determination of this period in the case when $\mu=0$. We also propose a variation of this result for general $\mu$ (see Question \ref{PeriodConjecture}) and give some computational support for this proposal.

\subsection{Outline of the paper} In Section \ref{GITbasics}, we collect the relevant definitions and results from GIT which will be needed for the remainder of the paper. We include some basic exposition (with references) for the convenience of the reader, but make no effort to be complete. In Section \ref{FlagGIT}, we apply these methods to the case of flag varieties, building up the proof of Theorem \ref{MainTheorem1}, with the key geometric ingredient being the asymptotic Riemann--Roch theorem. Finally, in Section \ref{Descent-Examples} we discuss the descent lattice for GIT quotients of flag varieties by line bundles, its relationship to Proposition \ref{MainTheorem2}, and give some explicit examples of the stretched Kostka quasi-polynomials for various types.

\subsection*{Acknowledgements} 
We thank Shrawan Kumar for helpful discussions on GIT quotients of flag varieties, and for outlining the argument of Proposition \ref{nonemptystable}, which greatly simplified and improved our exposition.

\section{Essentials of Geometric Invariant Theory} \label{GITbasics}

In this section, we collect the key features and results of geometric invariant theory (GIT) that will be used in this paper. We do not venture to give a complete treatment of this broad and well-established field; instead, we refer to the standard classical reference \cite{GIT} or the approachable treatment as in \cite{Dolgachev} for full details. We limit ourselves presently to the case when $X$ is a smooth, irreducible, complex projective variety with an action of a connected reductive algebraic group $H$ (to distinguish from the group $G$ as in the introduction), although we will later specialize to the case of a maximal torus $H=T$, which is the only relevant case for us.

\subsection{Stability and GIT quotients} 

Let $X$ and $H$ be as above, and let $\mathbb{L}$ be an ample, $H$-linearized line bundle on $X$. Recall the following standard definitions. 

\begin{definition} A point $x \in X$ is called 
\begin{enumerate}
\item \emph{semistable} with respect to $\mathbb{L}$ if there exists some $n >0$ and section $\sigma \in \mathrm{H}^0(X, \mathbb{L}^{\otimes n})^H$ such that $\sigma(x) \neq 0$. We denote the set of all such semistable points by $X^{ss}(\mathbb{L})$. 
\item \emph{stable} with respect to $\mathbb{L}$ if it is semistable, the isotropy subgroup $H_x$ is finite, and all $H$-orbits in $X_\sigma:=\{y \in X: \sigma(y) \neq 0\}$ are closed. We denote the set of all such stable points by $X^{s}(\mathbb{L})$; note that naturally $X^{s}(\mathbb{L}) \subseteq X^{ss}(\mathbb{L})$. 
\item \emph{unstable} with respect to $\mathbb{L}$ if it is not semistable. 

\end{enumerate}
\end{definition}

In general, taking a naive quotient of the variety $X$ by its $H$-action is poorly behaved. However, the utility of GIT is that we can get much better control on quotients by instead considering the sets $X^{ss}(\mathbb{L})$ or $X^{s}(\mathbb{L})$ of semistable or stable points, respectively, for an ample line bundle. This is made precise in the following construction and theorem regarding the GIT quotient (cf. \cite{Dolgachev}*{Theorem 8.1, Proposition 8.1}).


\begin{theorem} \label{GITquotient} For $X$, $H$, and $\mathbb{L}$ as above, set 
$$
X //_{\mathbb{L}} H:= \mathrm{Proj} \left( \bigoplus_{n \geq 0} \mathrm{H}^0(X, \mathbb{L}^{\otimes n})^H \right),
$$
the \emph{GIT quotient} of $X$ by $H$. Then 
\begin{enumerate}
\item $X //_{\mathbb{L}} H$ is a projective variety, and the map $\pi: X^{ss}(\mathbb{L}) \to X //_{\mathbb{L}} H$ associated to the inclusion 
$$\bigoplus_{n \geq 0} \mathrm{H}^0(X, \mathbb{L}^{\otimes n})^H \hookrightarrow \bigoplus_{n \geq 0} \mathrm{H}^0(X, \mathbb{L}^{\otimes n})$$ 
is a good categorical quotient. 

\item There is an ample line bundle $\mathbb{M}$ on $X //_{\mathbb{L}} H$ such that $\pi^\ast(\mathbb{M}) = \mathbb{L}^{\otimes m}|_{X^{ss}(\mathbb{L})}$ for some $m \geq 0$. 

\item There is an open subset $W \subseteq X //_{\mathbb{L}} H$ such that $X^{s}(\mathbb{L})= \pi^{-1}(W)$, and the restriction $\pi: X^{s}(\mathbb{L}) \to W$ is a geometric quotient of $X^{s}(\mathbb{L})$ by $H$. 

\end{enumerate}
\end{theorem}

By virtue of $\pi(X^{s}(\mathbb{L}))$ being a geometric quotient and open in the GIT quotient, Theorem \ref{GITquotient}(3) has the following corollary, which will be of key importance for us. 

\begin{corollary} \label{GITdim} Suppose that $X^{s}(\mathbb{L})$ is nonempty. Then $\dim \left(X //_{\mathbb{L}} H\right) = \dim(X)-\dim(H)$.
\end{corollary}

\subsection{Hilbert--Mumford stability criterion} Let $\delta: \C^\ast \to H$ be a one-parameter subgroup (OPS) of $H$. Then since $X$ is projective, for any OPS $\delta$ and $x \in X$ the limit point 
$$
x_0:=\lim_{t \to 0} \delta(t).x
$$
exists in $X$. Via $\delta$, the fiber $\mathbb{L}_{x_0}$ has an induced $\C^\ast$-action, given by some integer $r$ such that $\delta(t).z = t^rz$ for $t \in \C^\ast$, $z \in \mathbb{L}_{x_0}$. Following the conventions of \cite{GIT}, we set 
$$
\mu^{\mathbb{L}}(x, \delta):= r,
$$
the Mumford number for the pair $(x, \delta)$. The Hilbert--Mumford criterion, given in the following proposition, gives a concrete numerical condition for determining the (semi)stability of a point $x \in X$ with respect to $\mathbb{L}$. 

\begin{proposition} \label{HilbMum}
A point $x \in X$ is semistable (respectively, stable) with respect to $\mathbb{L}$ if and only if $\mu^{\mathbb{L}}(x, \delta) \geq 0$ (respectively, $\mu^{\mathbb{L}}(x, \delta) >0$) for all non-constant OPSs $\delta$. 
\end{proposition}

We highlight a setting when there is a particularly nice realization of the Mumford number $\mu^{\mathbb{L}}(x, \delta)$, following \cite{KP}*{\S 3}; this construction will be a crucial part of applying this machinery to flag varieties in the later sections. Let $V$ be a finite-dimensional $H$-representation and let $i: X \hookrightarrow \mathbb{P}(V)$ be an $H$-equivariant embedding. Set $\mathbb{L}:= i^\ast(\mathcal{O}(1))$. For an OPS $\delta$, let $\{e_1, \dots, e_n\}$ be an eigenbasis of $V$, so that 
$$
\delta(t) \cdot e_j = t^{\delta_j}e_j
$$
for $j=1, \dots, n$. Then for $x \in X$, we have the following proposition \cite{GIT}*{\S 2.1, Proposition 2.3}.

\begin{proposition} \label{computeMumford}
Write $i(x)=\left[ \sum_{j=1}^n x_j e_j \right] \in \mathbb{P}(V)$. Then for $\mathbb{L} = i^\ast(\mathcal{O}(1))$ as above, 
$$
\mu^{\mathbb{L}}(x, \delta) = \max_{j: x_j \neq 0} (-\delta_j).
$$
\end{proposition}

\subsection{Descent and pushforward of line bundles to GIT quotients}

As in Theorem \ref{GITquotient}, let $\pi$ be the quotient map $X^{ss}(\mathbb{L}) \to X//_{\mathbb{L}} H$. For $\mathcal{F}$ a quasi-coherent $H$-equivariant sheaf over $X^{ss}(\mathbb{L})$, let $\pi^H_\ast \mathcal{F}$ be the invariant direct image sheaf on $X//_{\mathbb{L}} H$, where the sections over an open set $U$ are the $H$-invariant sections in $\mathcal{F}(\pi^{-1}(U))$. Then $\pi^H_\ast$ is exact, and $(\pi^\ast, \pi^H_\ast)$ form an adjoint pair with $\pi^H_\ast \circ \pi^\ast = \mathrm{Id}$. We will be most interested in the case when $\mathcal{F}$ is a line bundle. 

\begin{definition} \label{descdef} Recall that a line bundle $\mathbb{L}'$ on $X$ \emph{descends} to $X//_{\mathbb{L}} H$ if there is a line bundle $\mathbb{M}'$ on $X//_{\mathbb{L}} H$ such that 
$$
\mathbb{L}'|_{X^{ss}(\mathbb{L})} = \pi^{\ast}(\mathbb{M}').
$$
\end{definition}

For example, in Theorem \ref{GITquotient}(2) there exists some $m \geq 0$ such that $\mathbb{L}^{\otimes m}$ descends to $X//_{\mathbb{L}} H$, and in some cases $m>1$ even though $\mathbb{L}$ was used to define the quotient. Via Kempf's \emph{descent lemma}, one can determine when a line bundle (or more generally, a vector bundle) descends to the GIT quotient; while we will later rely on descent conditions specific to the flag variety which follow from Kempf's descent lemma, we will not need the full result here, and instead refer to \cite{DN}*{Theorem 2.3} for details. 


By definition of the invariant direct image, if $\mathbb{L}$ descends to a line bundle $\mathbb{M}'$ on the GIT quotient, then $\pi^H_\ast(\mathbb{L})=\mathbb{M}'$. If $\mathbb{L}$ does \emph{not} descend, then the direct image $\pi^H_\ast(\mathbb{L})$ (if nonzero) is nonetheless a rank one reflexive sheaf on the GIT quotient, and in either case, the following result of Teleman \cite{Tel}*{Theorem 3.2.a and Remark 3.3(i)} allows us to compare the cohomology of these two sheaves. 

\begin{proposition} \label{cohomcompare}
$\mathrm{H}^\ast(X, \mathbb{L})^H = \mathrm{H}^\ast(X //_{\mathbb{L}} H, \pi^H_\ast(\mathbb{L})).$
\end{proposition}

\section{Quotients of flag varieties by maximal torus and stretched Kostka quasi-polynomials} \label{FlagGIT}

We now apply the machinery developed in the previous section to the case when $X=G/P$ is a flag variety and $H=T$ is a maximal torus; the role of this section is to assemble the proof of Theorem \ref{MainTheorem1}. To this end, we fix a simple, simply-connected complex algebraic group $G$ and a pair of primitive dominant integral weights $\mu \leq \lambda$, where we recall that by primitive we mean
$$
\lambda-\mu=\sum_i c_i \alpha_i
$$
as a sum of the simple roots of $G$ with each $c_i>0$. 

Write $\lambda=\sum_i d_i \varpi_i$ in the basis of fundamental weights, and as in Theorem \ref{MainTheorem1} set $\Phi'$ to be the sub-root system determined by $\Phi'=\mathrm{span}\{\alpha_i: d_i=0\} \subset \Phi$. Associated to this sub-root system is a standard parabolic subgroup $B \subset P_\lambda \subset G$ whose Levi component has root system $\Phi'$. We will abuse notation and write $P=P_\lambda$, as $\lambda$ will be fixed. Then $\lambda$ extends as a character of $P$ with trivial action on $U_P$. 

\subsection{Line bundles on $G/P$ and the dimension of a GIT quotient} We consider the $G$-equivariant line bundle $L_\lambda:= G \times_P \C_{-\lambda}$ on the flag variety $G/P$ associated to the principal $P$-bundle $G \to G/P$, where the choice of sign is for convenience. Then in particular $L_\lambda$ is an ample line bundle on $G/P$. Alternatively, via the embedding 
$$
i_\lambda: G/P \hookrightarrow \mathbb{P}(V_\lambda), \ gP \mapsto [gv_\lambda],
$$
where $v_\lambda \in V_\lambda$ is a nonzero highest-weight vector, we have that $L_\lambda \cong i_\lambda^\ast \mathcal{O}(1)$. By the classical Borel--Weil theorem, as $G$-representations  
$$
\mathrm{H}^0(G/P, L_\lambda) \cong V_\lambda^\ast,
$$
where $V_\lambda^\ast$ is the dual representation to $V_\lambda$, and all higher cohomology vanishes. 

We now modify the line bundle $L_\lambda$ to better suit our purposes; this we do by a change in its $T$-linearization. Specifically, set 
\begin{equation} \label{linebundef}
\mathbb{L}=\mathbb{L}_{\lambda, \mu}:= L_\lambda \otimes \C_\mu,
\end{equation}
where we again abuse notation and drop the dependence of $\mathbb{L}$ on $\lambda$ and $\mu$. The total space of this line bundle is still given by $G \times_P \C_{-\lambda}$, but the $T$-action is now given by 
$$
t.[g,z]:=[tg, \mu(t)z].
$$
Now applying the Borel--Weil theorem, as $T$-representations we get the identifications 
$$
\mathrm{H}^0(G/P, \mathbb{L}) \cong \C_\mu \otimes V_\lambda^\ast \cong (\C_{-\mu} \otimes V_\lambda)^\ast.
$$
In particular, by considering $T$-invariants we have that $\mathrm{H}^0(G/P, \mathbb{L})^T \neq 0 \iff V_\lambda(\mu) \neq 0$. With this perspective, we can now consider the nonvanishing of weight spaces in $V_\lambda$ via the language of GIT. More to the point, we have

\begin{proposition} \label{nonemptystable} 
Let $G/P$ and $\mathbb{L}$ be as above. Then $(G/P)^{s}(\mathbb{L}) \neq \varnothing$.
\end{proposition}

\begin{proof}
Fix a weight basis $\{v_i\}$ of the representation $V_\lambda$, and the corresponding dual basis $\{v_i^*\}$. Define subsets $X_i \subseteq G/P$ by 
$$
X_i:=\{gP: v_i^*(gv_\lambda)\ne 0\}.
$$
As $V_\lambda$ is irreducible, each $X_i$ is open and nonempty. Set $\mathring{X}:= \bigcap_i X_i$; this is also open and nonempty, since $G/P$ is irreducible. We claim that any $\mathring{g}P \in \mathring{X}$ is stable with respect to $\mathbb{L}$. 

Indeed, by Proposition \ref{computeMumford}, for any non-constant OPS $\delta: \mathbb{C}^\ast \to T$ we have 
$$
\mu^{\mathbb{L}}(\mathring{g}P, \delta) = \max_{\gamma \in \op{wt}(V_\lambda)} \{\langle -\gamma, \dot{\delta} \rangle\} +\langle \mu, \dot{\delta}\rangle,
$$
where $\op{wt}(V_\lambda)$ is the set of weights of $V_\lambda$ and $\langle \cdot, \cdot \rangle$ is the unique (up to scalar) non-degenerate $W$-invariant pairing between the weight and coweight lattice. Let $w \in W$ be such that $\dot{\delta}_+:= -w\dot{\delta}$ is dominant (by which we mean $\langle \alpha_i, \dot{\delta}_+ \rangle \geq 0$ for all simple roots $\alpha_i$). Since the set of weights $\op{wt}(V_\lambda)$ is $W$-stable and the pairing $\langle \cdot, \cdot \rangle$ is $W$-invariant, we get that 
$$
\begin{aligned}
\mu^\mathbb{L}(\mathring{g}P, \delta) &= \max_{\gamma \in \op{wt}(V_\lambda)} \{ \langle \gamma, \dot{\delta}_+ \rangle\} + \langle -w\mu, \dot{\delta}_+ \rangle \\
&= \langle \lambda, \dot{\delta}_+ \rangle - \langle w\mu, \dot{\delta}_+ \rangle \\ 
&= \langle \lambda-\mu, \dot{\delta}_+ \rangle + \langle \mu-w\mu, \dot{\delta}_+ \rangle
\end{aligned}
$$
Now, since $(\lambda, \mu)$ was assumed to be a primitive pair, necessarily $\langle \lambda-\mu, \dot{\delta}_+ \rangle >0$. And, since $\mu$ is dominant, we have that $\mu \geq w\mu$, so that $\langle \mu-w\mu, \dot{\delta}_+ \rangle \geq 0$. In all, this gives that $\mu^\mathbb{L}(\mathring{g}P, \delta) >0$, so by the Hilbert--Mumford criterion of Proposition \ref{HilbMum}, $\mathring{g}P$ is stable with respect to $\mathbb{L}$.  
\end{proof}

By Corollary \ref{GITdim}, we get the following immediate crucial corollary. 

\begin{corollary} \label{dimensioncount}
For $G/P$ and $\mathbb{L}$ as above, 
$$
\dim \left( (G/P) //_{\mathbb{L}} T \right) = \dim \left(G/P \right) -\dim \left( T \right) =  \frac{1}{2}|\Phi|-\frac{1}{2}|\Phi'| - \mathrm{rank}(\Phi),
$$
where as before $\Phi'$ is the root system for the Levi component of $P$. 
\end{corollary}

Thus, the desired degree of the quasi-polynomial $K_{\lambda, \mu}(N)$ for a primitive pair $(\lambda, \mu)$ has a geometric realization as the dimension of an associated GIT quotient. We now make this connection precise.

\subsection{Asymptotic Riemann-Roch and the degree of $K_{\lambda, \mu}(N)$}

Let $X$ be a projective variety and $\mathcal{F}$ a coherent sheaf on $X$. We denote the Euler--Poincar\'e characteristic of $\mathcal{F}$ by 
$$
\chi(\mathcal{F}):= \sum_{i \geq 0} (-1)^i \dim \mathrm{H}^i(X, \mathcal{F}).
$$

Now, consider the additional data of an ample line bundle $\mathbb{L}$ on $X$. The following proposition, often called the \emph{asymptotic Riemann--Roch formula}, is also a particular case of a theorem of Snapper \cite{Kleim}*{\S1 Theorem} and is closely connected with the notion of a Hilbert polynomial for a sheaf. This result is the key connection between the degree of $K_{\lambda, \mu}(N)$ and the dimension of the GIT quotient. We give a version of the statement which is most closely aligned to our purposes; the conditions can be somewhat relaxed. 

\begin{proposition} \label{aRR}
For $X$ a projective variety, $\mathcal{F}$ a coherent sheaf on $X$ and $\mathbb{L}$ an ample line bundle on $X$, the function 
$$
n \mapsto \chi(\mathcal{F} \otimes \mathbb{L}^{\otimes n})
$$
is a polynomial in $n$ of degree $\dim \op{Supp}(\mathcal{F})$. 
\end{proposition}

We apply this to the case of $X=(G/P)//_{\mathbb{L}} T$. As before, set $\mathbb{L} = L_\lambda \otimes \C_\mu$, where as always $(\lambda, \mu)$ is a primitive pair. As in the discussion after Definition \ref{descdef}, it need not be the case that $\mathbb{L}$ descends to a line bundle on $(G/P)//_{\mathbb{L}} T$. Nevertheless, some positive power does descend. Fix $k \geq 1$ such that $\mathbb{L}^{\otimes k}$ descends to the GIT quotient (for now, any such a choice suffices; we will return to this discussion in the next section). Denote by $\mathbb{M}$ the corresponding line bundle on the GIT quotient; note that $\mathbb{M}=\pi_\ast^T \left(\mathbb{L}^{\otimes k}\right)$ is ample.

Fixing $0\leq j <k$, we set $\mathcal{F}_j$ to be the coherent sheaf on $(G/P)//_\mathbb{L} T$ given by 
$$
\mathcal{F}_j:= \pi_\ast^T(\mathbb{L}^{\otimes j}),
$$
where $\pi_\ast^T$ is the invariant direct image. Recall that this is a reflexive sheaf of rank one. The following proposition is now a direct consequence of Proposition \ref{aRR}. 

\begin{proposition} \label{Kquasipiece}

For $G$ simple, simply-connected, $(\lambda, \mu)$ a primitive pair, and for a choice of $k \geq 1$ and $0 \leq j < k$ as above, the function $q_j: \mathbb{N} \to \mathbb{N}$ given by 
$$
q_j: n \mapsto \dim V_{(j+kn)\lambda}((j+kn)\mu)
$$
is a polynomial in $n$ of degree $\dim ((G/P)//_{\mathbb{L}} T)$. 
\end{proposition}

\begin{proof}
Let $\mathbb{L}=L_\lambda \otimes \C_\mu$ as above, with $k \geq 1$ chosen so that $\mathbb{L}^{\otimes k}$ descends to $\mathbb{M}$ on $(G/P)//_\mathbb{L} T$. Then by Proposition \ref{aRR}, for the sheaf $\mathcal{F}_j$ and $\mathbb{M}$, we have that 
$$
n \mapsto \chi(\mathcal{F}_j \otimes \mathbb{M}^{\otimes n})
$$
is a polynomial in $n$ of degree $\dim \op{Supp}(\mathcal{F}_j)$; since $\mathcal{F}_j$ is reflexive, it has full support so that this polynomial is of degree $\dim ((G/P) //_{\mathbb{L}} T)$. Now, since $\mathbb{M}=\pi_\ast^T(\mathbb{L}^{\otimes k})$ is a line bundle, we have by the projection formula that 
$$
\mathcal{F}_j \otimes \mathbb{M}^{\otimes n} = \pi_*^T( \mathbb{L}^{\otimes j} \otimes \pi^* (\mathbb{M}^{\otimes n}))=\pi_\ast^T (\mathbb{L}^{\otimes (j+kn)}).
$$

By Proposition \ref{cohomcompare}, we get that 
$$
\mathrm{H}^\ast((G/P)//_{\mathbb{L}} T, \mathcal{F}_j \otimes \mathbb{M}^{\otimes n}) \cong \mathrm{H}^\ast(G/P, \mathbb{L}^{\otimes (j+kn)})^T.
$$
By the higher cohomology vanishing of the Borel--Weil theorem, the only nontrivial terms are the global sections, and so
$$
\dim \mathrm{H}^0((G/P)//_{\mathbb{L}} T, \mathcal{F}_j \otimes \mathbb{M}^{\otimes n}) = \dim \mathrm{H}^0(G/P, \mathbb{L}^{\otimes (j+kn)})^T = \dim V_{(j+kn)\lambda}((j+kn)\mu)
$$
is a polynomial in $n$ of degree $\dim ((G/P)//_{\mathbb{L}} T)$, as desired. 
\end{proof}

Finally, we can collect these polynomials and arrange them into a quasi-polynomial, completing the proof of Theorem \ref{MainTheorem1}. 

\begin{theorem} \label{MT1Restated}
For $G$ simple, simply-connected and $(\lambda, \mu)$ a primitive pair, the function $K_{\lambda, \mu}(N)$ is quasi-polynomial of degree $\dim ((G/P)//_{\mathbb{L}} T)$. 
\end{theorem} 

\begin{proof}
As in Proposition \ref{Kquasipiece}, set $q_j(n)=\dim V_{(j+kn)\lambda}((j+kn)\mu)$, a polynomial in $n$, for $0 \leq j < k$. Set $p_j(N):=q_j(\frac{N-j}{k})$, which is a polynomial in $N$ of the same degree. Then if $N \equiv j  \mod k$, writing $N=j+nk$ we get that 
$$
K_{\lambda, \mu}(N) = \dim V_{N\lambda}(N\mu)= \dim V_{(j+nk)\lambda}((j+nk)\mu) = q_j(n) = q_j\left(\frac{N-j}{k}\right) = p_j(N),
$$
so that $p_0(N), p_1(N), \dots, p_{k-1}(N)$ defines a quasi-polynomial for $K_{\lambda, \mu}(N)$ of period $k$ and degree $\dim ((G/P)//_{\mathbb{L}} T)$, as each individual $p_i(N)$ has this degree. 
\end{proof}

We end this section with a few remarks. As previously discussed, Theorem \ref{MT1Restated} with the restriction to primitive pairs $(\lambda, \mu)$ was proven by Gao and Gao \cite{GG} in classical types using the combinatorics of Berenstein--Zelevinsky polytopes and Ehrhart theory. However, using similar combinatorics along with polytopes coming from the path model, Dehy \cite{Dehy} also derives the quasi-polynomality of the dimensions of weight spaces $V_{N \lambda}(N\mu)$ in $N$; this is done even more generally in the setting of Demazure modules for symmetrizable Kac--Moody algebras. Applying the results of \cite{Dehy} to the present case, for any $(\lambda, \mu)$ not necessarily primitive, one can conclude that $K_{\lambda, \mu}(N)$ is quasi-polynomial of degree bounded by that in Theorem \ref{MT1Restated}. This is also briefly explained therein in more geometric language (attributed to M. Brion), although with fewer details.

The key advantage to primitive $(\lambda, \mu)$ is that they realize the \emph{maximum} possible degree among the stretched Kostka quasi-polynomials. Along with Proposition \ref{nonemptystable} which proved this fact, we mention that an alternative view of this result is that the primitive condition on $(\lambda, \mu)$ ensures that the line bundle $\mathbb{L}$ is in the relative interior of the \emph{$T$-ample cone} $\mathcal{C}^{T}(G/P_\lambda)$ (c.f. \cite{DH}).

\section{Descent of line bundles on flag varieties and periods of stretched Kostka quasi-polynomials} \label{Descent-Examples}

The proof of Theorem \ref{MT1Restated} establishes the quasi-polynomality of $K_{\lambda, \mu}(N)$, with period $k$, for any choice of $k \geq 1$ such that $\mathbb{L}^{\otimes k}$ descended to the GIT quotient. Often, when considering a quasi-polynomial $f(N)$, one would like to know the \emph{minimal} such period. In the case of Kostka quasi-polynomials, a natural candidate for this would be the minimal power of $\mathbb{L}$ which descends to the GIT quotient. 

By way of motivation, in \cite{Kumar1} Kumar determines precisely for which dominant weights $\lambda$ the line bundle $L_\lambda$ descends to the GIT quotient $(G/P_\lambda)//_{L_\lambda} T$. These correspond to a certain lattice $\Gamma$ depending on $G$, which we refer to as the \emph{descent lattice}. This was later used in by Kumar--Prasad \cite{KP} to study the quasi-polynomality of zero weight spaces as the highest weight $\lambda$ varies in the root lattice $Q$. We recall the classification of these lattices $\Gamma$ in the following proposition.

\begin{proposition} \label{DescentLattice}

Let $\lambda \in X^\ast(T)^+$ be a dominant integral weight. Then the line bundle $L_\lambda$ descends to the GIT quotient $(G/P_\lambda) //_{L_\lambda} T$ if and only if $\lambda \in \Gamma$, where $\Gamma$ is the lattice

\begin{itemize}
\item[(a)] $Q$, if $G$ is of type $A_\ell$ ($\ell \geq 1$).
\item[(b)] $2Q$, if $G$ is of type $B_\ell$ ($\ell \geq 3$).
\item[(c)] $\Z 2\alpha_1+\cdots + \Z 2 \alpha_{n-1} + \Z \alpha_n$, if $G$ is of type $C_\ell$ ($\ell \geq 2$).
\item[(d1)] $\{n_1 \alpha_1+2n_2\alpha_2 + n_3 \alpha_3+n_4 \alpha_4 : n_i \in \Z \text{ and } n_1+n_3+n_4 \in 2\Z \}$, if $G$ is of type $D_4$.
\item[(d2)] $\{2n_1\alpha_1 +2n_2\alpha_2+\cdots+2n_{\ell-2} + n_{\ell-1}\alpha_{\ell-1} + n_\ell \alpha_\ell : n_i \in \Z \text{ and } n_{\ell-1}+n_\ell \in 2\Z\}$, if $G$ is of type $D_\ell$ ($\ell \geq 5$).
\item[(e)] $\Z 6 \alpha_1 + \Z 2 \alpha_2$, if $G$ is of type $G_2$. 
\item[(f)] $\Z6\alpha_1+\Z6\alpha_2+\Z12\alpha_3+\Z12\alpha_4$, if $G$ is of type $F_4$.
\item[(g)] $6X^\ast(T)$, if $G$ is of type $E_6$.
\item[(h)] $12X^\ast(T)$, if $G$ is of type $E_7$.
\item[(i)] $60Q$, if $G$ is of type $E_8$.
\end{itemize}

\end{proposition}

Note that for $G$ simple, the pair $(\lambda, 0)$ is always primitive for any dominant integral weight $\lambda$, as the inverse of a simple Cartan matrix has strictly positive entries (cf. \cite{LT}). Applying this result, we have the following corollary concerning periods of stretched Kostka quasi-polynomials $K_{\lambda, 0}(N)$.

\begin{corollary} \label{MT2Restated}
Let $\lambda$ be a dominant integral weight with $\lambda \geq 0$. Let $d \geq 1$ be the minimal integer such that $d\lambda \in \Gamma$. Then the stretched Kostka quasi-polynomial $K_{\lambda, 0}(N)$ has period $d$.

 In particular, if $G=SL_n(\C)$, then this recovers the $\mu=0$ case of the polynomality of stretched Kostka numbers due to Kirillov--Reshetikhin \cite{KR}, and if $G$ is of classical type, then $K_{\lambda, 0}(N)$ is quasi-polynomial of period at most $2$.
\end{corollary}

%

Following Corollary \ref{MT2Restated}, we can ask two natural questions: first, is the period $d$ produced above always the \emph{minimal} period for the quasi-polynomial $K_{\lambda, 0} (N)$, and second, is there a similar easy determination of a period for general pairs $(\lambda, \mu)$? To answer the second question, one would like to first have a handle on (at least sufficient) conditions for the line bundle $\mathbb{L}:=L_\lambda \otimes \C_\mu$ to descend to the GIT quotient, in terms of $\lambda$ and $\mu$. This seems more subtle than the $\mu=0$ case, and remains an open question of interest. For the first question (even in the case of general pairs), while the minimal integer $d$ such that the line bundle $\mathbb{L}^{\otimes d}$ descends to the GIT quotient is a natural candidate of minimal period, we do not know how to explicitly rule out a smaller divisor of $d$ being a period of $K_{\lambda, \mu}(N)$. 

Nevertheless, as in \cite{GG}*{Proposition 3.4, Corollary 3.5}, we can always reduce these questions to the case of primitive pairs by taking appropriate projections onto the weight lattices of certain Levi subgroups. For general $(\lambda, \mu)$, a period for $K_{\lambda, \mu}(N)$ would then be given by the least common multiple of the periods of each of its primitive factors. Using this reduction, we offer the following question, which is based in part on a hopeful analogy with Corollary \ref{MT2Restated}, as a candidate for a period of $K_{\lambda, \mu}(N)$ which generalizes the $\mu=0$ case. We include in the subsequent subsections some computational examples as support of an affirmative answer to this question. 

\begin{question} \label{PeriodConjecture}
Let $G$ be a simple, simply-connected complex algebraic group and $(\lambda, \mu)$ a primitive pair of dominant integral weights. Let $d \geq 1$ be the minimal integer such that $d(\lambda-\mu) \in \Gamma$. Then does the stretched Kostka quasi-polynomial $K_{\lambda, \mu}(N)$ have $d$ as a (minimal) period? 
\end{question}

\subsection{Examples and computational method}
We conclude with some explicit instructive examples, which were computed via Sage \cite{TSD}. All results are simply informed guesses at the quasi-polynomial $K_{\lambda,\mu}(N)$. From basic linear algebra, we know that any $n$ points can be interpolated by a degree $n-1$ polynomial. It is relatively much ``harder'' (measure $0$ probability) for those $n$ points to be interpolated by a degree $n-2$ polynomial. 

Now supposing we have an integer sequence $(a_N)$ which is known to be the outputs of a polynomial function $f(N)=a_N$, with $\deg(f)$ unknown. If the first $n$ points of the sequence are interpolated by a degree $n-2$ polynomial, it is suggestive that the same polynomial predicts all terms of the sequence. This is, of course, not a guarantee, but still compelling evidence. 

In the calculations reported below, a sample of $k$ points of the sequence $K_{N\lambda,N\mu}$  was collected using {\tt LiE} \cite{lie}. From there, for each possible choice of positive integer $d$ such that $k/d\ge 2$, we split the terms of the sequence up into $d$ groups of at least $n=\lfloor k/d \rfloor$ terms each according to the residue class of the index modulo $d$. If each group of $n$ terms defines a polynomial of degree at most $n-2$, we decide there is compelling evidence that a quasi-polynomial of period $d$ and degree $n$ defines the sequence which we have sampled. If no such quasi-polynomial could be found, we increase $k$ and try again.  We emphasize that, in the computation of these examples, we do \emph{not} give as input the degree of the quasi-polynomial $K_{\lambda, \mu}(N)$ as determined by Theorem \ref{MainTheorem1}. That the computational methods produce quasi-polynomials of the correct degree and proposed period (and not smaller) is further support for Question \ref{PeriodConjecture}.

\subsubsection{An example from type $G_2$}

Let $G$ be of type $G_2$, $\lambda = \varpi_1+3\varpi_2$, and $\mu = \varpi_2$. Then $\lambda - \mu = 8 \alpha_1 + 5\alpha_2$, and the smallest multiple of $\lambda - \mu$ that lands in $\Gamma$ is $6(\lambda - \mu)$. 

For the first $k=36$ data points, we found that $K_{\lambda,\mu}(N)$ agrees with the quasipolynomial defined by 

$$
\left\{
\begin{array}{rc}

1   + 3N  + \frac{205}{36}N^2 + \frac{160}{27}N^3 + \frac{61}{27}N^4, & N \equiv 0  \mod 6 \\ \\
\frac{17}{108}  + \frac{53}{27}N +\frac{205}{36}N^2 + \frac{160}{27}N^3 + \frac{61}{27}N^4, & N \equiv 1  \mod 6 \\ \\
\frac{19}{27}  + \frac{67}{27}N + \frac{205}{36}N^2 + \frac{160}{27}N^3 + \frac{61}{27}N^4, & N \equiv 2  \mod 6 \\ \\
\frac{3}{4}  + 3N + \frac{205}{36}N^2 + \frac{160}{27}N^3 + \frac{61}{27}N^4, & N \equiv 3  \mod 6 \\ \\
\frac{11}{27}  + \frac{53}{27}N + \frac{205}{36}N^2 + \frac{160}{27}N^3 + \frac{61}{27}N^4, & N \equiv 4  \mod 6 \\ \\
\frac{49}{108}  + \frac{67}{27}N + \frac{205}{36}N^2 + \frac{160}{27}N^3 + \frac{61}{27}N^4, & N \equiv 5  \mod 6 \\ 

\end{array}
\right.
$$
~\\

This quasipolynomial has degree $4$, as it must given that $\Phi' = \emptyset$ and 
$$
\frac{1}{2} |\Phi| - \mathrm{rank}(\Phi) = \frac{1}{2}(12)-2 = 4.
$$

\subsubsection{An example from type $B_3$}

Let $G = \mathrm{Spin}(7)$, $\lambda = \varpi_1+\varpi_2+\varpi_3$, and $\mu = \varpi_3$. Then $\lambda - \mu = 2\alpha_1 + 3\alpha_2 + 3\alpha_3$. While $\lambda - \mu \not\in \Gamma$, $2(\lambda - \mu) \in \Gamma$, so we anticipate a period of $2$ to the quasipolynomial. 

For the first $k=16$ data points, the sequence $K_{\lambda,\mu}(N)$ agrees with the function 

$$
\left\{
\begin{array}{rc}
1 + \frac{113}{40}N + \frac{2057}{480}N^2 + \frac{63}{16}N^3 + \frac{859}{384} N^4 + \frac{457}{640} N^5 + \frac{91}{960} N^6, & N \equiv 0 \mod 2 \\ \\ 
\frac{85}{128} + \frac{1553}{640}N + \frac{961}{240}N^2 + \frac{247}{64}N^3 + \frac{859}{384} N^4 + \frac{457}{640} N^5 + \frac{91}{960} N^6, & N \equiv 1 \mod 2 \\ \\ 
\end{array}
\right. 
$$

The degree of this quasipolynomial is $6$, and given that $\Phi' = \emptyset$, this is the correct degree, for 
$$
\frac{1}{2}|\Phi| - \mathrm{rank}(\Phi) = \frac{1}{2}(18) - 3 = 6.
$$

\subsubsection{An example from type $D_4$}

Let $G = \mathrm{Spin}(8)$, $\lambda = \varpi_2$, and $\mu = 0$. We have $\lambda - \mu = \alpha_1+2\alpha_2+\alpha_3+\alpha_4$. Note that $2(\lambda-\mu) \in \Gamma$, and $d = 2$ yields the smallest multiple to get into the descent lattice. 

For the first $k=14$ data points, we have agreement of $K_{\lambda,\mu}(N)$ with the quasipolynomial

$$
\left\{
\begin{array}{rc}
1 + \frac{101}{60}N + \frac{25}{24} N^2 + \frac{5}{16} N^3 + \frac{5}{96} N^4 + \frac{1}{240} N^5, & N \equiv 0 \mod 2 \\\\
\frac{29}{32} + \frac{101}{60}N + \frac{25}{24} N^2 + \frac{5}{16} N^3 + \frac{5}{96} N^4 + \frac{1}{240} N^5, & N \equiv 1 \mod 2 \\
\end{array}
\right.
$$
~\\

The degree is $5$, which is correct given that $\Phi'$ is of type $A_1^{\times 3}$ and hence 
$$
\frac{1}{2}|\Phi| - \frac{1}{2}|\Phi'| - \mathrm{rank}(\Phi) = \frac{1}{2}(24) - \frac{1}{2}(6) - 4 = 5.
$$

\subsubsection{An example from type $F_4$}

Let $G$ be of type $F_4$, $\lambda = 2 \varpi_1$, and $\mu = 0$. We have $\lambda - \mu = 4 \alpha_1 +6\alpha_2 +8\alpha_3 +4\alpha_4$. Note that $3(\lambda-\mu) \in \Gamma$, and $d=3$ is the smallest multiple to achieve this. 

The quasipolynomial $K_{\lambda,\mu}(N)$ agrees with the functions $\sum c_i N^i$ defined by the table below, depending on the remainder of $N$ modulo $3$, to the first $k=39$ data points. 

$$
\begin{array}{r|rrrrrrrrrrrr}
 & c_0 & c_1 & c_2 & c_3 & c_4 & c_5 & c_6 & c_7 & c_8 & c_9 & c_{10} & c_{11} \\ \hline \\
 N \equiv 0 \mod 3 & 1 & \frac{4111}{1386}    & \frac{87911}{18900}  & \frac{555223}{113400}  & \frac{18533}{4860} 
 & \frac{4397}{1944} & \frac{6184}{6075} & \frac{4828}{14175} & \frac{232}{2835} & \frac{68}{5103} & \frac{8}{6075} & \frac{4}{66825} \\ \\

  N\equiv 1 \mod 3 & \frac{697}{729} & \frac{4111}{1386}    & \frac{87911}{18900}  & \frac{555223}{113400}  & \frac{18533}{4860} 
 & \frac{4397}{1944} & \frac{6184}{6075} & \frac{4828}{14175} & \frac{232}{2835} & \frac{68}{5103} & \frac{8}{6075} & \frac{4}{66825} \\ \\
 N\equiv 2 \mod 3 & \frac{665}{729} & \frac{4111}{1386}    & \frac{87911}{18900}  & \frac{555223}{113400}  & \frac{18533}{4860} 
 & \frac{4397}{1944} & \frac{6184}{6075} & \frac{4828}{14175} & \frac{232}{2835} & \frac{68}{5103} & \frac{8}{6075} & \frac{4}{66825} \\ 
\end{array}
$$
~\\

The degree of the quasi-polynomial is $11$, as it must be. Note that $\Phi'$ is the Levi root system on nodes $2$, $3$, and $4$, hence of type $C_3$ with $|\Phi'| = 18$. The degree is therefore
$$
\frac{1}{2} |\Phi| - \frac{1}{2} |\Phi'| - \mathrm{rank}(\Phi) = \frac{1}{2}(48) - \frac{1}{2}(18) - 4 = 11.
$$

\end{document}